\documentclass{amsart}

\usepackage[english]{babel}

\usepackage{amsmath}
\usepackage{amssymb}
\usepackage{mathrsfs}
\usepackage{graphicx}
\usepackage[colorlinks=true, allcolors=blue]{hyperref}
\usepackage{enumitem}
\usepackage{xfrac}
\usepackage{todonotes}

\newtheorem{lemma}{Lemma}
\newtheorem{theorem}{Theorem}

\title[Infinitely many quasi--arithmetic maximal reflection groups]{Infinitely many quasi--arithmetic \\ maximal reflection groups}

\author{Edoardo Dotti}
\address{Dipartimento Formazione e Apprendimento, Piazza San Francesco 19, 6600 Locarno, Svizzera / Switzerland}
\email{edoardo.dotti@supsi.ch}
\thanks{E.D. was partially supported by Swiss National Science Foundation (project PP00P2-170560)}

\author{Alexander Kolpakov}
\address{Institut de Math\'ematiques, Rue Emile--Argand 11, 2000 Neuch\^atel, Suisse / Switzerland}
\email{kolpakov.alexander@gmail.com}
\thanks{A.K. was partially supported by Swiss National Science Foundation (projects PP00P2-170560 and PP00P2-202667)}

\begin{document}

\begin{abstract}
In contrast to the fact that there are only finitely many maximal arithmetic reflection groups acting on the hyperbolic space $\mathbb{H}^n$,~$n\geq 2$, we show that:

\begin{itemize}[leftmargin = 1.5pc, rightmargin = 1.5pc]
    \item[(a)] one can produce infinitely many maximal quasi--arithmetic reflection groups acting on $\mathbb{H}^2$;
    \item[(b)] they admit infinitely many different fields of definition;
    \item[(c)] the degrees of their fields of definition are unbounded.
\end{itemize}
However, for $n\geq 14$ an approach initially developed by Vinberg shows that there are still finitely many fields of definitions in the quasi--arithmetic case. 
\end{abstract}

\maketitle

\section{Introduction}\label{sec:intro}

One of the main results in the theory of hyperbolic reflection groups is that there are only finitely many maximal arithmetic reflection groups. This was independently proved by Agol, Belolipetsky, Storm, and Whyte \cite{absw}, and by Nikulin \cite{nikulin}. Both proofs use the fact that there are no arithmetic reflection groups acting on $\mathbb{H}^n$ with $n\geq 30$, as was shown by Vinberg \cite{vinberg-1984}. The techniques of \cite{absw} essentially use spectral bounds for the Laplacian operator, while the approach of \cite{nikulin} follows in spirit the previous work of Nikulin and Vinberg. 

In particular, Nikulin shows that whenever $\Gamma < \mathrm{Isom}\,\mathbb{H}^n$, $n \geq 2$, is an arithmetic reflection group defined over $k$, the degree of $k$ is bounded \cite{nikulin}. Together with the works \cite{bel-fields, bel-lin-fields, linowitz, mac}, this result implies that $\mathrm{deg}_{\mathbb{Q}}\,k \leq 25$.
Thus, it is natural to ask whether these results can be extended to quasi--arithmetic reflection groups, since these enjoy some of the fundamental properties of arithmetic groups.
However, the following theorem shows that in the quasi--arithmetic setting we have a very different picture (Section~\ref{sec:proof}). 

\begin{theorem}\label{theorem:main}
For every integer $d \geq 1$ there exists a totally real number field $k$ of degree $d$ over $\mathbb{Q}$ and a maximal quasi--arithmetic lattice $\Gamma < \mathrm{Isom}\, \mathbb{H}^2$ generated by reflections, with its field of definition exactly~$k$. 
\end{theorem}

Moreover, for any fixed number field $k$ of given degree $d \geq 1$ there are infinitely many quasi--arithmetic maximal reflection groups $\Gamma$ defined over a totally real number field $k$ acting on $\mathbb{H}^2$ (Section~\ref{sec:comments}).

Interestingly enough, the situation in higher dimensions is closer to the arithmetic setting. Let $K_1$ and $K_2$ denote the following sets of number fields:
\begin{equation*}
    K_1 = \{ \mathbb{Q}(\sqrt{2}), \mathbb{Q}(\sqrt{5}), \mathbb{Q}(\cos \sfrac{2\pi}{7})\},
\end{equation*}
and 
\begin{equation*}
    K_2 = \{ \mathbb{Q}(\sqrt{3}), \mathbb{Q}(\sqrt{6}), \mathbb{Q}(\sqrt{2},\sqrt{3}), \mathbb{Q}(\sqrt{2},\sqrt{5})\} \cup \{ \mathbb{Q}(\cos \sfrac{2\pi}{m}) \}_{m = 9, 11, 15, 16, 20}.
\end{equation*}

The following theorem was proved by Vinberg in \cite{vinberg-1984} for arithmetic reflection groups.   
\begin{theorem}\label{theorem:bounds}
If $\Gamma < \mathrm{Isom}\, \mathbb{H}^n$ is a cocompact quasi--arithmetic reflection group defined over a totally real number field $k$, then $k \in K_1$ for $n \geq 22$, and $k \in K_1 \cup K_2$ for $n \geq 14$. 
\end{theorem}

By following the proof given in \cite{vinberg-1984} one can make sure that \textit{quasi--arithmeticity} is already sufficient for Theorem~\ref{theorem:bounds} to hold (Section~\ref{sec:dimensions}). Together, Theorem~\ref{theorem:main} and Theorem~\ref{theorem:bounds} provide a partial answer to Question~7.5 in \cite{bk}.

\section*{Acknowledgements}

The authors are grateful to Nikolay V. Bogachev (IITP RAS \& CombGeoLab MIPT, Russia) for useful remarks and stimulating discussions. The authors also thank the anonymous referee for her or his comments that helped substantially improving the paper. 

\section{Preliminaries}\label{sec:prel}

Let $\mathbb{H}^n$, $n\geq 2$, be the hyperbolic $n$--dimensional space. We shall use the \textit{hyperboloid model} of $\mathbb{H}^n$ as described in \cite[\S~3.2]{ratcliffe}. Namely, let the two--sheeted hyperboloid $\mathscr{H}$ be given by

\begin{equation*}
    \mathscr{H} = \{ x \in \mathbb{R}^{n+1} \,|\, \langle x, x \rangle = -1 \},
\end{equation*}
where
\begin{equation*}
    \langle x,y \rangle = - x_0 y_0 + x_1 y_1 + \ldots  + x_n y_n
\end{equation*}
is the \textit{Lorentzian product} of $x = (x_0, \ldots, x_n) \in \mathscr{H}$ and $y = (y_0, \ldots, y_n) \in \mathscr{H}$. Let $\mathbb{R}^{n,1}$ denote the vector space $\mathbb{R}^{n+1}$ equipped with the above Lorentzian product.

Let $\mathscr{H}^+ = \mathscr{H} \cap \{ x \in \mathbb{R}^{n+1} \,|\, x_0 > 0 \}$ be the upper sheet of $\mathscr{H}$. Then $\mathscr{H}^+$ with the distance $d(x,y)$ defined by
\begin{equation*}
    \cosh d(x,y) = - \langle x, y \rangle
\end{equation*}
is a Riemannian space of constant sectional curvature $-1$. Let $\mathbb{H}^n$ denote this latter space. A detailed account of different properties of $\mathbb{H}^n$ is given in \cite{ratcliffe}.

Let $\mathrm{O}_{n,1}$ be the orthogonal group of the above Lorentzian product, and $\mathrm{O}^+_{n,1}$ be its index $2$ subgroup preserving $\mathscr{H}^+$. The isomorphisms  $\mathrm{Isom}\,\mathbb{H}^n \cong \mathrm{O}^+_{n,1} \cong \mathrm{PO}_{n,1}$ are well--known.

Let $\Gamma < \mathrm{Isom}\, \mathbb{H}^n$ be a discrete subgroup. Then $\Gamma$ is called a \textit{lattice} (or a \textit{finite covolume} subgroup) if a fundamental set for the action of $\Gamma$ on $\mathbb{H}^n$ has finite volume. We shall also call $\Gamma$ a \textit{hyperbolic lattice} for short. If $\Gamma$ acts on $\mathbb{H}^n$ with a compact fundamental set, then $\Gamma$ is called a \textit{uniform lattice} (or a \textit{cocompact} subgroup). 

Let $P \subset \mathbb{H}^n$ be a convex polytope. If $\Gamma < \mathrm{Isom}\, \mathbb{H}^n$ is generated by reflections in the facets (i.e. codimension $1$ faces) of $P$, then $\Gamma$ is called a \textit{hyperbolic reflection group}. In this case $\Gamma$ is a lattice whenever $P$ has finite volume, and $\Gamma$ is uniform whenever $P$ is compact.

If the dihedral angles of $P$ are integer submultiples of $\pi$ (i.e. each angle has the form $\sfrac{\pi}{k}$, for some integer $k\geq 2$), then $P$ is called a \textit{Coxeter polytope}. Let $\{e_i\}^m_{i=1}$ be the set of unit outer normals in $\mathbb{R}^{n,1}$ to the facets of $P$. The \textit{Gram matrix} of $P$ is $G(P) = \{ \langle e_i, e_j \rangle \}^m_{i,j=1}$. It is worth mentioning that there are no lattices generated by reflections acting on $\mathbb{H}^n$ if $n\geq 997$ \cite{khovanskii, prokhorov}. There are no uniform lattices generated by reflections acting on $\mathbb{H}^n$ if $n\geq 30$ \cite[Theorem 1]{vinberg-1984}.

The so--called \textit{arithmetic subgroups} of $\mathrm{PO}_{n,1}$ form an important class of hyperbolic lattices. Here and below we shall only deal with reflection groups that are arithmetic of the simplest type (and related objects), that were studied in \cite{vinberg-1967} where a very clear description is given. 

Let $\Gamma < \mathrm{Isom}\, \mathbb{H}^n$ be a reflection group corresponding to a Coxeter polytope $P \subset \mathbb{H}^n$ of finite volume. Let $G = G(P)$ be its Gram matrix. 

The following result due to Vinberg turns out to be very practical.

\begin{theorem}[``Vinberg's criterion'', \cite{vinberg-1967}]\label{theorem:vinberg}
Let $P \subset \mathbb{H}^n$ be a Coxeter polytope of finite volume, and let $G = G(P)$ be its Gram matrix. Let $K = K(G)$ be the field generated over $\mathbb{Q}$ by the entries of $G$, and $k = k(G)$ be the field generated over $\mathbb{Q}$ by all cyclic products of entries of $G$. Let $\Gamma = \Gamma(P)$ be the reflection group associated with $P$. Then $\Gamma$ is an arithmetic lattice if and only if
\begin{itemize}
    \item[(a)] $K$ is a totally real number field,
    \item[(b)] for all embeddings $\sigma: K \rightarrow \mathbb{R}$ such that $\sigma|_k \neq \mathrm{id}$ the matrix $G^\sigma$ is positive semi--definite,
    \item[(c)] all cyclic products of entries of $2 G$ belong to the ring of integers $\mathscr{O}_k$ of $k$.
\end{itemize}
\end{theorem}

Thus, if the reflection group $\Gamma$ above satisfies conditions (a), (b) and (c), we call $\Gamma$ an \textit{arithmetic reflection group}. An arithmetic reflection group is always a lattice. 

If $\Gamma$ above is a lattice that satisfies conditions (a) and (b), but not necessarily~(c), then we call $\Gamma$ a \textit{quasi--arithmetic} lattice. If only (a) and (b) are satisfied while (c) fails then $\Gamma$ is called a \textit{properly} quasi--arithmetic lattice \cite{emery, vinberg-1967}.

Two subgroups $\Gamma_1,$ $\Gamma_2 < \mathrm{Isom}\, \mathbb{H}^n$ are called \textit{commensurable} if $\Gamma_1 \cap \Gamma_2$ has finite index in both $\Gamma_1$ and $\Gamma_2$. Same two subgroups are called \textit{commensurable in the wide sense} if there exists an element $g \in \mathrm{Isom}\, \mathbb{H}^n$ such that $\Gamma_1$ and $g \Gamma_2 g^{-1}$ are commensurable. 

If $\Gamma$ is a lattice in $\mathrm{PO}_{n,1}$, $n\geq 2$, then it always has a field of definition $k$ as shown in \cite{vinberg-1971}. We shall assume that $k$ is the smallest field containing the traces of the adjoint representation $\mathrm{Ad}\, \Gamma$. 

For any given lattice $\Gamma$, its field of definition $k$ is invariant under commensurability. Assume that $\Gamma$ is defined over a number field. As shown by Vinberg in \cite{vinberg-1971}, the \textit{ring of definition} $R_\Gamma$ of $\Gamma$ is its commensurability invariant as well. Such a ring of definition is characterised as the smallest ring $R_\Gamma < k$ containing the traces of $\mathrm{Ad}\,\Gamma$. Note that for a quasi--arithmetic reflection group $\Gamma = \Gamma(P)$ as in Theorem~\ref{theorem:vinberg}, $k$ corresponds to the field generated by the cycles of the Gram matrix $G = G(P)$ and is always a number field. Moreover, the ring of definition $R_\Gamma$ corresponds to the ring of integers $\mathscr{O}_k$ adjoint all the cycles of $2G$, cf. \cite[Proposition 4.3.11]{dotti}.

\section{Proof of Theorem~\ref{theorem:main}}\label{sec:proof}

Let $\mathscr{Q} = \mathscr{Q}(\varphi, \ell)$ be the Lambert quadrilateral that has three right angles (encountered in a sequence while traversing its boundary in the anticlockwise direction) and one angle $\varphi$. Let $\ell$ be the length of one of its sides, and $\theta = \cosh \ell$, as marked in Figure~\ref{fig:Lambert}. Then $\mathscr{Q}(\varphi, \ell)$ is determined up to isometry by the parameters $\varphi$~and~$\ell$. 

The Gram matrix of $\mathscr{Q} = \mathscr{Q}(\varphi, \ell)$ turns out to be
\begin{equation*}
    G(\mathscr{Q}) = \begin{pmatrix}
    1 & 0 & -\theta & 0\\
    0 & 1 & 0 & -\rho\\
    -\theta & 0 & 1 & -\cos \varphi\\
    0 & -\rho & -\cos \varphi & 1
    \end{pmatrix},
\end{equation*}
where $\rho$ is the hyperbolic cosine of the other side's length (see Figure \ref{fig:Lambert}). Notice that, since $G(\mathscr{Q})$ has vanishing determinant, we have $\rho = \sqrt{\frac{\theta^2-\sin^2 \varphi}{\theta^2-1}}$. Set $\mathscr{Q}_\theta = \mathscr{Q}(\sfrac{\pi}{3}, \cosh^{-1} \theta)$. Then the reflections in the sides of $\mathscr{Q}_\theta$ generate a cocompact discrete reflection subgroup $\Gamma_\theta < \mathrm{Isom}\, \mathbb{H}^2$. The field of definition of $\Gamma_\theta$ is generated by $\theta^2$: we have $k(\Gamma_\theta) = \mathbb{Q}(\theta^2)$. The field $K(\Gamma_\theta)$ is generated by $\theta$ and $\rho$, and thus $K(\Gamma_\theta) = \mathbb{Q}(\theta, \rho)$. 

\begin{figure}[ht]
    \centering
    \includegraphics[scale=0.25]{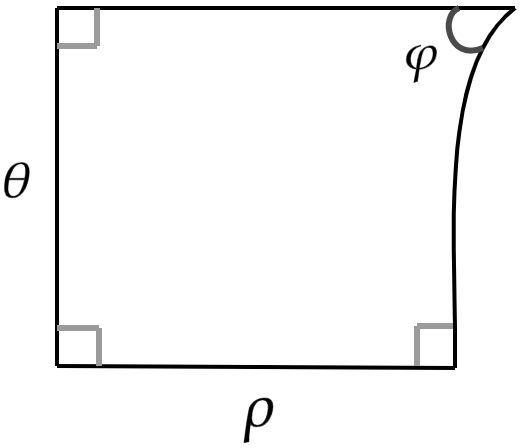}
    \caption{The Lambert quadrilateral $\mathscr{Q}$}
    \label{fig:Lambert}
\end{figure}

Let $\mathscr{P}$ denote the set of totally real Pisot numbers. Then the following theorem describes one of the important roles of $\mathscr{P}$: its elements generate all possible totally real number fields. It is also worth mentioning that Pisot numbers were used before in a related setting of discrete subgroups of higher-rank semisimple Lie groups \cite{bl}. Such subgroups are known to be always arithmetic \cite{margulis}.  

\begin{theorem}[T. Za\"{i}mi, \cite{zaimi}]\label{theorem:zaimi}
Let $k$ be a totally real number field of degree $d$. Then $k$ is generated over $\mathbb{Q}$ by a Pisot number $\tau \in \mathscr{P}$ such that $\tau \geq \alpha^{d-1}_0$, and $\alpha^4_0$ is the maximal root of $x^5 - 5x^4 - x^3 + 5x^2 - 1$.
\end{theorem}

Let us set $\mathscr{T} = \left\{ \frac{71}{82}\, \tau \,|\, \tau \in \mathscr{P}  \right\}$. Here we use $\sfrac{71}{82}$ as a rational approximation (from below) for $\sfrac{\sqrt{3}}{2}$ that is sufficiently good for our purposes: 
$| \frac{71}{82} - \frac{\sqrt{3}}{2} | < 10^{-3}$. 

\begin{lemma}\label{lemma:1}
If $\theta \in \mathscr{T}$, then the field $K(\Gamma_\theta)$ is totally real.
\end{lemma}
\begin{proof}
Since $\theta$ is already a totally real number, it remains to show that $\rho$ is totally real, too. We have that 
\begin{equation*}
    \rho^2 = \frac{\theta^2 - \sfrac{3}{4}}{\theta^2 - 1}. 
\end{equation*}
For all $\theta \in \mathscr{T}$ we have $\theta \geq \sfrac{71}{82} \cdot \alpha_0 > 1.29 \ldots$ by Theorem \ref{theorem:zaimi}. This yields $\rho \in \mathbb{R}$. Moreover, for any non-trivial Galois automorphism $\sigma \in \mathrm{Gal}(K/\mathbb{Q})$ we have that $\sigma(\theta) < \sfrac{71}{82} < \sfrac{\sqrt{3}}{2}$. Thus $\sigma(\rho) \in \mathbb{R}$ for all $\sigma \in \mathrm{Gal}(K/\mathbb{Q})$.  
\end{proof}

\begin{lemma}\label{lemma:3}
If $\theta \in \mathscr{T}$, then the group $\Gamma_\theta$ is quasi--arithmetic.
\end{lemma}
\begin{proof}
By Lemma~\ref{lemma:1}, the field $K = K(\Gamma_\theta)$ is totally real. So is the field of definition $k = k(\Gamma_\theta) < K$. Let $\sigma$ be a non--trivial Galois automorphism of $K$ such that $\sigma|_k \neq \mathrm{id}$. We need to check that all principal minors of $G^\sigma$ are non--negative. The latter condition is easily reduced to the inequality $\sigma(\theta)^2 < \sfrac{3}{4}$. By construction, for every $\theta \in \mathscr{T}$ we have $\sigma(\theta)^2 < \sfrac{5041}{6724} < \sfrac{3}{4}$. Hence, by Vinberg's criterion (Theorem~\ref{theorem:vinberg}), $\Gamma_\theta$ is quasi--arithmetic.
\end{proof}

\begin{lemma}\label{lemma:2}
The group $\Gamma_\theta$ is maximal whenever it is properly quasi--arithmetic. 
\end{lemma}

\begin{proof}
Assume that $\Gamma_\theta$ is \textit{properly} quasi--arithmetic. First, let us show that $\Gamma_\theta$ is a \textit{maximal reflection group} in the following sense: for any subgroup $\Gamma < \mathrm{Isom}\, \mathbb{H}^2$ generated by reflections we have that $\Gamma_\theta < \Gamma$ with finite index if and only if $\Gamma = \Gamma_\theta$. 

As follows from the Riemann--Hurwitz formula (see also \cite[Theorem 1]{felikson-tumarkin1} and \cite[Theorem 1.2]{felikson-tumarkin2}), $\Gamma$ is generated by either $3$ or $4$ reflections. Since $\Gamma$ acts on $\mathbb{H}^2$ cocompactly, we have that $\Gamma$ is either a triangle reflection group, or $\Gamma$ is generated by reflections in the sides of a quadrilateral. 

\medskip

\paragraph{\it Case (i): $\Gamma$ is a triangle group.} In this case, if $\Gamma$ is quasi--arithmetic then it is automatically arithmetic. This follows easily from Vinberg's criterion (Theorem~\ref{theorem:vinberg}) if we recall that $2 \cos(\sfrac{\pi}{m})$ is an algebraic integer for all $m \geq 1$. However, as we suppose that $\Gamma_\theta$ is \textit{properly quasi--arithmetic}, it cannot be commensurable with $\Gamma$ in this case since arithmeticity is a commensurability invariant. 

\medskip

\begin{figure}[ht]
    \centering
    \includegraphics[scale=0.125]{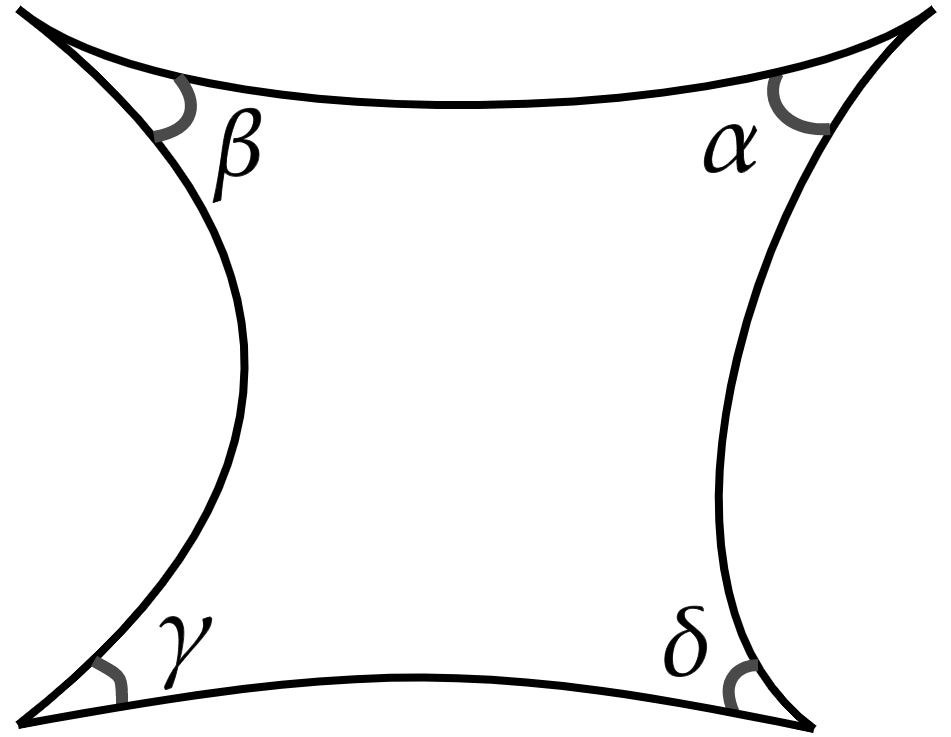}
    \caption{A hyperbolic quadrilateral}
    \label{fig:quad}
\end{figure}

\paragraph{\it Case (ii): $\Gamma$ is a quadrilateral group.} Let $P$ be the fundamental polygon for $\Gamma$ that is a quadrilateral with angles $\alpha$, $\beta$, $\gamma$, $\delta$ pictured in Figure~\ref{fig:quad}. Such a quadrilateral exists whenever $\alpha + \beta + \gamma + \delta < 2\pi$. Moreover, since $P$ is a Coxeter polygon, we have that $\alpha, \beta, \gamma, \delta \leq \sfrac{\pi}{2}$ and all of them are integer submultiples of $\pi$. Then
\begin{equation*}
    \mathrm{area}\, (\mathbb{H}^2/\Gamma) = 2\pi - \alpha - \beta - \gamma - \delta \geq \frac{\pi}{6} = \mathrm{area}\, (\mathbb{H}^2/\Gamma_\theta). 
\end{equation*}
Thus we have $|\Gamma:\Gamma_\theta| = {\mathrm{area}\,(\mathbb{H}^2/\Gamma_\theta)}/{\mathrm{area}\,(\mathbb{H}^2/\Gamma)} = 1$, which implies that $\Gamma = \Gamma_\theta$.

\medskip

Thus, we obtain that $\Gamma_\theta$ is a maximal reflection group with fundamental polygon $P = \mathscr{Q}(\sfrac{\pi}{3}, \cosh^{-1} \theta)$. By \cite[Proposition 3]{vinberg-1972} any lattice $\Gamma$ containing $\Gamma_\theta$ with finite index satisfies $\Gamma = \Gamma_\theta \rtimes \mathfrak{S}$, where $\mathfrak{S}$ is the symmetry group of $P$. Observe that $\mathfrak{S}$ is trivial unless $\theta = \rho = \sqrt{\sfrac{3}{2}}$. In the latter case, however, $\Gamma_\theta$ is easily seen to be arithmetic by Vinberg's criterion (Theorem~\ref{theorem:vinberg}). 
\end{proof}

The group $\Gamma_\theta$ is quasi--arithmetic whenever $\theta \in \mathscr{T}$, as follows from Lemma~\ref{lemma:3}. Even more, in this way we obtain infinitely many \textit{properly} quasi--arithmetic groups. Indeed, the degree $d = |\mathbb{Q}(\theta):\mathbb{Q}|$ can be arbitrarily high, and the degree of $k = \mathbb{Q}(\theta^2)$ is at least $\sfrac{d}{2}$. Since the degree of the field of definition of an arithmetic group is bounded (see \cite{absw, bel-fields, bel-lin-fields, linowitz, mac, nikulin}) we obtain that $\Gamma_\theta$ cannot be arithmetic for $d > 50$. By combining this fact with Lemma~\ref{lemma:2} we obtain Theorem~\ref{theorem:main}. 

\section{Remarks: groups and fields}\label{sec:comments}

There are a few things that are easy to show, and nevertheless we would like to stress them in order to provide a complete picture of how flexible maximal quasi--arithmetic reflection subgroups of $\mathrm{Isom}\, \mathbb{H}^2$ can be. 

\medskip

\paragraph{\it Fixing the degree of the field $k$.} There are infinitely many totally real fields $k$ of degree $d$, for each $d\geq 2$. Thus, we obtain infinitely many pairwise incommensurable groups $\Gamma_\theta$ with $k = \mathbb{Q}(\theta^2)$. If $k = \mathbb{Q}$, then we set $\theta = \sqrt{1+p}$, for a rational prime~$p$. The groups $\Gamma_\theta$ will always be properly quasi--arithmetic, since 
\begin{equation*}
    4 \rho^2 = 4 + \frac{1}{\theta^2-1} = 4 + \frac{1}{p} \notin \mathbb{Z}.
\end{equation*}

Moreover, the ring of definition is given by $R(\Gamma_\theta) = \mathbb{Z}[\sfrac{1}{p}]$, and hence the groups $\Gamma_\theta$ are pairwise incommensurable for different rational primes $p$.

\medskip

\paragraph{\it Fixing the field $k$ itself.} If we bound both the degree and discriminant of $k$, then we have only a finite number of fields $k$ satisfying the bounds. Nevertheless, we can still obtain infinitely many pairwise incommensurable groups $\Gamma_\theta$. In this case, one can pick $\mathscr{Q}_\theta$ that produces a properly quasi--arithmetic reflection group defined over $k$, and then replace $\theta$ by $\tau = (1-\sfrac{1}{p}) \theta$, with $p$ being a sufficiently large rational prime so that $\Gamma_\tau$ satisfies Vinberg's criterion (Theorem~\ref{theorem:vinberg}). Provided that the degree of $\theta$ over $\mathbb{Q}$ is greater than $2$, we have that $R(\Gamma_{\tau})$ and $R(\Gamma_{\eta})$ are different whenever $\tau = (1-\sfrac{1}{p}) \theta$ and $\eta = (1-\sfrac{1}{q}) \theta$ with sufficiently large $p > q > 2$. In this case, the denominators of the elements in $R(\Gamma_{\tau})$ and $R(\Gamma_{\eta})$ have a different set of rational primes involved. If $\theta$ has degree $2$ over $\mathbb{Q}$, then the field of definition is $\mathbb{Q}$ and the construction in the previous paragraph can be applied. 

\section{Remarks: quadratic forms}

Let us observe that condition (b) of Vinberg's criterion (Theorem \ref{theorem:vinberg}) implies that a non--uniform quasi--arithmetic lattice generated by hyperbolic reflections always has $\mathbb{Q}$ as its field of definition. However, in contrast to the arithmetic case, some \textit{uniform} quasi--arithmetic lattices generated by reflections may preserve \textit{isotropic} quadratic forms. First such examples were given by Vinberg in \cite{Vinberg-Bielefeld}. 

Let us consider the Lambert quadrilateral $\mathscr{Q}_\theta$ with $\theta = \sqrt{1+3p^2}$, where $p > 3$ is a rational prime. The associated reflection group $\Gamma_\theta < \mathrm{Isom}\, \mathbb{H}^2$ is a uniform lattice. By applying Jacobi's theorem to the Gram matrix of $\mathscr{Q}_\theta$ one can easily find that $\Gamma_\theta$ preserves a quadratic form equivalent to 
\begin{equation*}
    f_\theta(x, y, z) = 36 p^2 x^2 + 27 p^2 y^2 - 4 (1 + 3 p^2) z^2.
\end{equation*}
This form is isotropic, as can be easily verified by setting $x = 1$, $y = 2p$, $z = 3p$. Moreover, we have infinitely many incommensurable such lattices $\Gamma_\theta$ for different primes $p > 3$. By the preceding discussion, all of them are properly quasi--arithmetic and thus maximal.

\section{Higher dimensions: Theorem \ref{theorem:bounds}}\label{sec:dimensions}

As pointed out in the introduction, one can use Vinberg's local determinants technique \cite[\S~5]{vinberg-1984} coupled with the dimension bounds in \cite[Proposition~2]{vinberg-1984} to show that the set of possible fields of definition of quasi--arithmetic reflection groups acting on $\mathbb{H}^n$ is finite for $n\geq 14$. This fact is stated by Vinberg as Theorem 2 and Theorem 3 in \cite{vinberg-1984} for arithmetic groups only. However, one can see that quasi--arithmeticity suffices for Vinberg's arguments to work. Since the details are rather technical, we recall here some basic steps.

The core of the proof consists of combinatorial techniques applied to Coxeter schemes of reflection groups, and their subschemes. This combinatorial aspect is independent of  arithmeticity. Moreover, let us recall that a non--uniform quasi--arithmetic lattice always has $\mathbb{Q}$ as its field of definition. Therefore, one has to focus only on cocompact reflection groups.

The peculiarity of cocompact reflection groups is that their Coxeter schemes always contain a Lann\'er subscheme (corresponding to a cocompact simplex reflection group). In \cite[Proposition 17]{vinberg-1984}, Vinberg shows that the determinant of any Lann\'er subscheme of the Coxeter scheme associated to an arithmetic cocompact reflection group generates its field of definition. The only conditions of Vinberg's arithmeticity criterion (Theorem~\ref{theorem:vinberg}) used in this proof are conditions (a) and (b). Hence, the above statement also applies to fields of definition of quasi--arithmetic reflection groups.

The aforementioned combinatorial techniques are then applied to Lann\'er subschemes whose determinant is a negative algebraic number with all other Galois conjugates being positive. These are Lann\'er subschemes of quasi--arithmetic reflection groups. Thus, Theorem~\ref{theorem:bounds} follows immediately from what has already been established in \cite{vinberg-1984}.

Unfortunately, we cannot obtain a similar result about the possible fields of definition for quasi--arithmetic reflection groups acting on $\mathbb{H}^n$, where $3 \leq n \leq 13$. A na\"ive conjecture would be that this list is finite for $n\geq 4$, which can be partly supported by the fact that lattices acting on $\mathbb{H}^n$ by isometries demonstrate very strong discreteness properties with respect to geometric convergence if $n\geq 4$ (cf. Wang's theorem \cite{wang}). For $n = 3$ we lack any kind of reasonable corroboration to an analogous statement or to its opposite. To the best of our knowledge, all known families of non--arithmetic reflection groups acting on $\mathbb{H}^3$ have only finitely many quasi--arithmetic members (e.g. reflection groups generated by triangular prisms cf. \cite{vinberg-1967}). Moreover, the approaches used by Agol, Belolipetsky, Storm and Whyte, and also by Nikulin, heavily rely on arithmeticity properties, and therefore do not provide information on maximality of quasi--arithmetic groups.

Finally, Vinberg showed the absence of arithmetic reflection groups acting on $\mathbb{H}^n$ for $n \geq 30$ \cite[Theorem 4]{vinberg-1984}. It is thus natural to ask whether this result can be transferred to quasi--arithmetic reflection groups. 

In order to obtain the above dimension bound, Vinberg first shows that there are no reflective Lorentzian lattices of rank $31$ or greater. This result is then combined with the following fact: every non--uniform \textit{arithmetic} lattice generated by reflections acting on the hyperbolic space of dimension $n\geq 2$ is a subgroup of finite index in $O(L)$, the group of orthogonal transformations of a rank $(n+1)$ Lorentzian reflective lattice $L$. However, the latter does not hold in general for quasi--arithmetic groups, which precludes a straightforward generalisation.

\bibliographystyle{abbrv}
\bibliography{biblio.bib}
\end{document}